\documentclass[11pt,reqno,oneside]{amsart}

\usepackage{enumerate}
\numberwithin{equation}{section}
\theoremstyle{definition}
\newtheorem{Definition}{Definition}[section]
\newtheorem{Example}[Definition]{Example}
\newtheorem{Remark}[Definition]{Remark}
\newtheorem{Problem}[Definition]{Problem}

\theoremstyle{plain}
\newtheorem{Theorem}[Definition]{Theorem}

\newtheorem{Corollary}[Definition]{Corollary}
\newtheorem{Lemma}[Definition]{Lemma}

\usepackage{geometry}

\usepackage{amssymb}        % This automatically loads amsfonts.
\usepackage{mathrsfs}       % \mathscr will use the Ralph Smith’s For­mal Script.
\usepackage{eucal}          % \mathcal will use the Euler Script.
\usepackage{stmaryrd}       % St Mary’s Road symbol font. Contains \llbracket, \rrbracket, \mapsfrom.

\usepackage{hyperref}

\usepackage[usenames,dvipsnames]{xcolor}
\usepackage{tikz}
\usetikzlibrary{arrows, matrix}
\usepackage{tikz-cd}
\usepackage{subfig}         % Enables subfigures
\usepackage{arydshln}       % Dash-lines in array/tabular

\newcommand{\al}{\alpha}
\newcommand{\be}{\beta}
\newcommand{\ga}{\mathrm{\gamma}}

\newcommand{\ep}{\varepsilon}

\newcommand{\la}{\lambda}

\newcommand{\si}{\sigma}

\newcommand{\Z}{\mathbb{Z}}

\newcommand{\K}{\Bbbk}

\newcommand{\Be}{\mathbf{e}}

\newcommand{\Fgl}{\mathfrak{gl}}
\newcommand{\Fsl}{\mathfrak{sl}}

\newcommand{\Fosp}{\mathfrak{osp}}

\newcommand{\Fg}{\mathfrak{g}}

\DeclareMathOperator{\Ann}{Ann}
\DeclareMathOperator{\Aut}{Aut}
\DeclareMathOperator{\Cliff}{Cliff}

\DeclareMathOperator{\Id}{Id}

\DeclareMathOperator{\sgn}{sgn}

\DeclareMathOperator{\Supp}{Supp}

\newcommand{\iv}[2]{\llbracket #1,#2 \rrbracket}

\renewcommand{\tilde}{\widetilde}

\newcommand{\TGWC}[4]{\mathcal{C}_{#1}({#2},{#3},{#4})}
\newcommand{\TGWA}[4]{\mathcal{A}_{#1}({#2},{#3},{#4})}
\newcommand{\TGWI}[4]{\mathcal{I}_{#1}({#2},{#3},{#4})}

\title{Clifford and Weyl superalgebras and spinor representations}

\author{Jonas T. Hartwig\and Vera Serganova}
\date{}

\address{Department of Mathematics, Iowa State University, Ames, IA-50011, USA}
\email{jth@iastate.edu}
\address{Department of Mathematics, University of California, Berkeley, CA-94720, USA}
\email{serganov@math.berkeley.edu}

\subjclass[2010]{17B35; 15A66}
\keywords{Lie superalgebra; spinor representation; generalized Weyl algebra}

\begin{document}
\begin{abstract}
We construct a family of twisted generalized Weyl algebras which includes Weyl-Clifford superalgebras and quotients of the enveloping algebras of $\mathfrak{gl}(m|n)$ and $\mathfrak{osp}(m|2n)$. We give a condition for when a canonical representation by differential operators is faithful. Lastly, we give a description of the graded support of these algebras in terms of pattern-avoiding vector compositions.
\end{abstract}
\maketitle

\section{Introduction}
Twisted generalized Weyl algebras (TGWAs) were introduced by Mazorchuk and Turowska in \cite{MazTur1999,MazTur2002} in an attempt to include a wider range of examples than Bavula's generalized Weyl algebras (GWAs) \cite{Bavula1992}.
 Their structure and representations have been studied in \cite{MazTur1999,MazTur2002,MazPonTur2003,Sergeev2001,Hartwig2006,Hartwig2010,Hartwig2016,HarOin2013,FutHar2012b}.
 Known examples of TGWAs include multiparameter quantized Weyl algebras \cite{MazTur2002,Hartwig2006,FutHar2012b}, the Mickelsson-Zhelobenko step algebras associated to $(\Fgl_{n+1},\Fgl_n\oplus\Fgl_1)$ \cite{MazPonTur2003} and some primitive quotients of enveloping algebras \cite{HarSer2016}.

In this paper we take a step further by proving that supersymmetric analogs of some classical algebras are also examples of TGWAs. Specifically, we show that Weyl-Clifford superalgebras and some quotients of the enveloping algebras of $\mathfrak{gl}(m|n)$ and $\mathfrak{osp}(m|2n)$ can be realized as twisted generalized Weyl (TGW) algebras. This suggests that much of the general representation theory from \cite{MazTur2002,MazPonTur2003,Hartwig2006} could be applied to the study of certain families of superalgebras. In addition our new algebras provide a large supply of consistent but non-regular TGW algebras (i.e. certain elements $t_i$ are zero-divisors). This motivates future development of the theory to include such algebras.

To summarize the contents of the present paper, in Section \ref{sec:tgwas} we recall the definition of TGW algebras from \cite{MazTur2002} which includes certain scalars $\mu_{ij}$ that in our case will be $\pm 1$. Some known results that will be used are also stated.
In Section \ref{sec:clifford-weyl} we prove that the Weyl-Clifford superalgebra from \cite{Nis1990} can be realized as a TGW algebra.

The main object of the paper is introduced in Section \ref{sec:family}, in which we define a family of TGW algebras $\mathcal{A}(\ga)^\pm$ which depend on a certain matrix $\ga$ with integer entries. These algebras naturally come with an algebra homomorphism $\varphi$ from $\mathcal{A}(\ga)^\pm$ to a Clifford-Weyl algebra. This is a generalization of the construction in \cite{HarSer2016}. A sufficient condition for $\varphi_\ga$ to be injective is given in Section \ref{sec:inj}. This condition is related to the graded support of the algebra $\mathcal{A}(\ga)^\pm$ which is combinatorially characterized in Section \ref{sec:support}.

Lastly, these results are applied in Section \ref{sec:Lie-superalgebras} to prove that for appropriate $\ga$, the TGW algebras $\mathcal{A}(\ga)^\pm$ fit into commutative diagrams involving the spinor representation $\pi$ of $U(\Fg)$ for $\Fg=\Fgl(m|n)$ and $\Fosp(m|2n)$ studied by Nishiyama \cite{Nis1990} and Coulembier \cite{Cou2013}. As a corollary we obtain that $U(\Fg)/J$ are examples of TGW algebras for such $\Fg$ as well as for classical Lie algebras. These results generalize previous realizations in \cite{HarSer2016}. We end with some open problems regarding exceptional types.

\subsection*{Notation}
Throughout, we work over an algebraically closed field $\K$ of characteristic zero. Associative algebras are assumed to have a multiplicative identity.
$\iv{a}{b}$ denotes the set of integers $x$ with $a\le x\le b$.

\section{Twisted generalized Weyl algebras} \label{sec:tgwas}
We recall the definition of TGW algebras and some of their useful properties.

\subsection{Definitions} \label{sec:definitions}
Let $I$ be a set.
\begin{Definition}[TGW Datum] 
A \emph{twisted generalized Weyl datum over $\K$ with index set $I$} is a triple $(R,\si,t)$ where
\begin{itemize}
\item $R$ is an associative $\K$-algebra,
\item $\si=(\si_i)_{i\in I}$ a sequence of commuting $\K$-algebra automorphisms of $R$,
\item $t=(t_i)_{i\in I}$ is a sequence of central elements of $R$.
\end{itemize} 
\end{Definition}

Let $\Z I$ denote the free abelian group on $I$, with basis denoted $\{\Be_i\}_{i\in I}$.
For $g=\sum g_i\Be_i\in\Z I$ put $\si_g=\prod \si_i^{g_i}$. Then $g\mapsto\si_g$ defines an action of $\Z I$ on $R$ by $\K$-algebra automorphisms.

\begin{Definition}[TGW Construction]
Let
\begin{itemize}
\item $(R,\si,t)$ be a TGW datum  over $\K$ with index set $I$,
\item $\mu$ be an $I\times I$-matrix
without diagonal, $\mu=(\mu_{ij})_{i\neq j}$,
with $\mu_{ij}\in\K\setminus\{0\}$.
\end{itemize}
The \emph{twisted generalized Weyl construction} associated to $\mu$ and $(R,\si,t)$, denoted $\TGWC{\mu}{R}{\si}{t}$, is defined as the free $R$-ring on the set $\{X_i,Y_i\mid i\in I\}$ modulo the two-sided ideal generated by the following elements:
\begin{subequations}\label{eq:tgwarels}
\begin{alignat}{3}
\label{eq:tgwarels1}
X_ir  &-\si_i(r)X_i,  &\qquad Y_ir&-\si_i^{-1}(r)Y_i, 
 &\qquad \text{$\forall r\in R,\, i\in I$,} \\
\label{eq:tgwarels2}
Y_iX_i&-t_i, &\qquad X_iY_i&-\si_i(t_i),
 &\qquad \text{$\forall i\in I$,} \\
\label{eq:tgwarels3}
&&\qquad X_iY_j&-\mu_{ij}Y_jX_i,
 &\qquad \text{$\forall i,j\in I,\, i\neq j$.}
\end{alignat}
\end{subequations}
\end{Definition}

The algebra $\TGWC{\mu}{R}{\si}{t}$ has a $\Z I$-gradation
given by requiring $\deg X_i=\Be_i, \deg Y_i=-\Be_i, \deg r=0\, \forall r\in R$.
Let $\TGWI{\mu}{R}{\si}{t}\subseteq \TGWC{\mu}{R}{\si}{t}$ be 
the sum of all graded ideals $J\subseteq \TGWC{\mu}{R}{\si}{t}$ 
such that $\TGWC{\mu}{R}{\si}{t}_0\cap J=\{0\}$.
It is easy to see that $\TGWI{\mu}{R}{\si}{t}$ is the unique maximal graded ideal having
zero intersection with the degree zero component.

\begin{Definition}[TGW Algebra]
The \emph{twisted generalized Weyl algebra} $\TGWA{\mu}{R}{\si}{t}$ associated to $\mu$ and $(R,\si,t)$ is
defined as the quotient $\TGWA{\mu}{R}{\si}{t}:=\TGWC{\mu}{R}{\si}{t} / \TGWI{\mu}{R}{\si}{t}$.
\end{Definition}
Since $\TGWI{\mu}{R}{\si}{t}$ is graded,
$\TGWA{\mu}{R}{\si}{t}$ inherits a $\Z I$-gradation from $\TGWC{\mu}{R}{\si}{t}$.
The images in $\TGWA{\mu}{R}{\si}{t}$ of the elements $X_i, Y_i$ will also be denoted by $X_i, Y_i$.

\begin{Example}\label{ex:WeylAlgebraExample}
For an index set $I$, the $I$:th Weyl algebra over $\K$, $A_I=A_I(\K)$ is the $\K$-algebra generated by $\{x_i,\partial_i\mid i\in I\}$ subject to defining relations
\[[x_i,x_j]=[\partial_i,\partial_j]=[\partial_i,x_j]-\delta_{ij}=0,\quad\forall i,j\in I.\]
There is a $\K$-algebra isomorphism $\TGWA{\mu}{R}{\tau}{u}\to A_n$ where $\mu_{ij}=1$ for all $i\neq j$, $R=\K[u_i\mid i\in I]$, $\tau_i(u_j)=u_j-\delta_{ij}$, given by $X_i\mapsto x_i$, $Y_i\mapsto \partial_i$, $u_i\mapsto \partial_ix_i$.
\end{Example}

\subsection{Regularity and consistency}
\label{sec:regcon}

\begin{Definition}[Reduced and monic monomials]
A \emph{monic monomial} in a TGW algebra is any finite product of elements from the set $\{X_i\}_{i\in I}\cup\{Y_i\}_{i\in I}$. A \emph{reduced monomial} is an elements of the form
$Y_{i_1}\cdots Y_{i_k} X_{j_1}\cdots X_{j_l}$
where $\{i_1,\ldots,i_k\}\cap \{j_1,\ldots,j_l\}=\emptyset$.
\end{Definition}

\begin{Lemma}{\cite[Lem.~3.2]{Hartwig2006}}
\label{lem:monomials}
$\TGWA{\mu}{R}{\si}{t}$ is generated as a left (and as a right) $R$-module by the reduced monomials.
\end{Lemma}

Since a TGW algebra $\TGWA{\mu}{R}{\si}{t}$ is a quotient of an $R$-ring, it is an $R$-ring itself with a natural map $\rho:R\to \TGWA{\mu}{R}{\si}{t}$. By Lemma \ref{lem:monomials}, the degree zero component of $\TGWA{\mu}{R}{\si}{t}$ (with respect to the $\Z I$-gradation) is equal to the image of $\rho$.

\begin{Definition}[Regularity]
A TGW datum $(R,\si,t)$ is called \emph{regular} if $t_i$
is regular (i.e. not a zero-divisor) in $R$ for all $i$.
\end{Definition}

Due to Relation \eqref{eq:tgwarels2}, the canonical map $R\to\TGWC{\mu}{R}{\si}{t}$ is not guaranteed to be injective, and indeed sometimes it is not \cite{FutHar2012b}. It is injective if and only if the map $R\to\TGWA{\mu}{R}{\si}{t}$ is injective. 

\begin{Definition}[$\mu$-Consistency]
A TGW datum $(R,\si,t)$ is \emph{$\mu$-consistent} if the canonical map $\rho:R\to \TGWA{\mu}{R}{\si}{t}$ is injective.
\end{Definition}

Abusing language we say that a TGW algebra $\TGWA{\mu}{R}{\si}{t}$ is regular (respectively consistent) if $(R,\si,t)$ is regular (respectively $\mu$-consistent).

\begin{Theorem}[{\cite{FutHar2012b}}]\label{thm:consistency}
A regular TGW algebra $\TGWA{\mu}{R}{\si}{t}$ is consistent iff
\begin{subequations}\label{eq:consistency_rels}
\begin{align}\label{eq:consistency_rel1}
\si_i\si_j(t_it_j)&=\mu_{ij}\mu_{ji}\si_i(t_i)\si_j(t_j),\quad\forall i\neq j;\\
\si_i\si_k(t_j)t_j&=\si_i(t_j)\si_k(t_j),\quad \forall i\neq j\neq k\neq i.
\end{align}
\end{subequations}
\end{Theorem}

That relation \eqref{eq:consistency_rel1} is necessary for consistency of a regular TGW datum was known already in \cite{MazTur1999,MazTur2002}. If $(R,\si,t)$ is not regular, sufficient and necessary conditions for $\mu$-consistency are not known (see Problem \ref{prb:consistency}). In this paper we produce many examples of consistent but non-regular TGW algebras.

Conversely, for consistent TGW algebras one can characterize regularity as follows:

\begin{Theorem}{\cite[Thm.~4.3]{HarOin2013}}
\label{thm:regularity}
Let $A=\TGWA{\mu}{R}{\si}{t}$ be a consistent TGW algebra. Then the following are equivalent
\begin{enumerate}[{\rm (i)}]
\item $(R,\si,t)$ is regular;
\item Each monic monomial in $A$ is non-zero and generates a free left (and right) $R$-module of rank one;
\item $A$ is \emph{regularly graded}, i.e. for all $g\in\Z I$, there exists a nonzero regular element in $A_g$;
\item If $a\in A$ is a homogeneous element such that $bac=0$ for some monic monomials $b,c\in A$, then $a=0$.
\end{enumerate}
\end{Theorem}

\subsection{Non-degeneracy of the gradation form}
For a group $G$, any $G$-graded ring $A=\bigoplus_{g\in G}A_g$
can be equipped with a $\Z$-bilinear form $\ga:A\times A\to A_e$ called the \emph{gradation form}, defined by
\begin{equation}
\ga(a,b)=\mathfrak{p}_e(ab)
\end{equation}
where $\mathfrak{p}_e$ is the projection $A\to A_e$ along the direct sum $\bigoplus_{g\in G} A_g$, and $e\in G$ is the neutral element.

\begin{Theorem}[{\cite[Cor.~3.3]{HarOin2013}}]
\label{thm:nondeg}
The ideal $\TGWI{\mu}{R}{\si}{t}$ is equal to the radical of the gradation form $\ga$ of $\TGWC{\mu}{R}{\si}{t}$ (with respect to the $\Z I$-gradation), and thus the gradation form on $\TGWA{\mu}{R}{\si}{t}$ is non-degenerate.
\end{Theorem}

\subsection{$R$-rings with involution}

\begin{Definition}\label{dfn:Rring_with_involution} 
\begin{enumerate}[{\rm (i)}]
\item An \emph{involution} on a ring $A$ is a $\Z$-linear map $\ast:A\to A, a\mapsto a^\ast$ satisfying $(ab)^\ast=b^\ast a^\ast$, $(a^\ast)^\ast=a$ for all $a,b\in A$.
\item Let $R$ be a ring.
An \emph{$R$-ring with involution} is a ring $A$ equipped with a ring homomorphism $h_A:R\to A$ and an involution $\ast:A\to A$ such that $h(r)^\ast=h(r)$ for all $r\in R$.
\item 
If $A$ and $B$ are two $R$-rings with involution, then a \emph{map of $R$-rings with involution} is a ring homomorphism $k:A\to B$ such that $k\circ h_A=h_B$ and $k(a^\ast)=(k(a))^\ast$ for all $a\in A$.
\end{enumerate}
\end{Definition}

Any TGW algebra $A=\TGWA{\mu}{R}{\si}{t}$ for which $\mu_{ij}=\mu_{ji}$ for all $i,j$, can be equipped with an involution $\ast$ given by $X_i^\ast=Y_i,\, Y_i^\ast=X_i\;\forall i\in I$, $r^\ast=r\;\forall r\in R$. Together with the canonical map $\rho:R\to A$ this turns $A$ into an $R$-ring with involution. In particular we regard the Weyl algebra $A_I$ as an $R$-ring with involution in this way, where $R=\K[u_i\mid i\in I]$ as in Example \ref{ex:WeylAlgebraExample}.

\section{The Clifford/Weyl superalgebras} \label{sec:clifford-weyl}

In this section let $\pm\in\{+,-\}$ and put $\mp=-\pm$.
Let $p$ and $q$ be non-negative integers and put $n=p+q$.
We consider supersymmetric analogs $A_{p|q}^\pm$ of Clifford and Weyl algebras and prove that they can be presented as TGW algebras.

\subsection{Definition and properties} \label{sec:clifford-weyl-definition}

\begin{Definition}
The \emph{Clifford/Weyl superalgebra of degree $p|q$}, denoted $A_{p|q}^\pm$, is defined as the superalgebra with even generators $x_i, \partial_i\; (i\in \iv{1}{p})$ and odd generators $x_i, \partial_i\; (i\in \iv{p+1}{n})$ and relations
\begin{equation}
[\partial_i,x_j]_\pm-\delta_{ij}=[x_i,x_j]_\pm=[\partial_i,\partial_j]_\pm=0\qquad\text{for all $i,j\in \iv{1}{n}$,}
\end{equation}
where $[\cdot,\cdot]_\pm$ denotes the super(anti-)commutator
\begin{equation}
[a,b]_\pm = ab\pm (-1)^{p(a)p(b)}ba.
\end{equation}
\end{Definition}
Thus $A_{p|q}^+$ (respectively $A_{p|q}^-$) is a supersymmetric analog of the Clifford (resepctively Weyl) algebra. 

We will need the following result, which is straightforward to verify.

\begin{Lemma}\label{lem:RpqMaxComm}
The subalgebra $R$ of $A_{p|q}^\pm$ generated by $\{\partial_i x_i\mid i\in \iv{1}{n}\}$ is maximal commutative.
\end{Lemma}

By the defining relations, $A_{p|q}^\pm$ is a graded algebra with respect to the free abelian group $\Z^n$. In addition $A_{p|q}^\pm$ has an involution $\ast$ given by $x_i^\ast=\partial_i$, $\partial_i^\ast=x_i$. Since $(\partial_i x_i)^\ast=\partial_i x_i$, $A_{p|q}^\pm$ is an $R$-ring with involution. Even though $A_{p|q}^\pm$ is not a domain in general, the following graded regularity property still holds.

\begin{Lemma}\label{lem:aaast1}
Let $a\in A_{p|q}^\pm$ be homogeneous of degree $g\in \Z^n$. If $a^\ast \cdot a=0$ then $a=0$.
\end{Lemma}

\begin{proof} We give a proof for $A=A_{p|q}^-$, the other case being analogous.
Write $a=rx^{(g)}$ where $r\in R$ and $x^{(g)}=x_1^{(g_1)}\cdots x_n^{(g_n)}$ where for $s>0$,  $x_i^{(s)}=x_i^s$, $x_i^{(-s)}=\partial_i^s$. By reordering the indices, we may assume that the first $k$ elements of the tuple $(g_{p+1},\ldots,g_{n})$ are zero, and the rest are nonzero. Put $u_i=\partial_i x_i$. For $i>p$ we have $u_i x_i=\partial_ix_i^2=0$ and $u_i\partial_i=\partial_ix_i\partial_i=(1-x_i\partial_i)\partial_i=\partial_i$. Thus we may assume that $r$ lies in the subalgebra of $R$ generated by $\{u_1,\ldots, u_{p+k}\}$.
If $a\cdot a^\ast=0$ then we have 
\begin{equation}\label{eq:lem_step}
0=a\cdot a^\ast = r x^{(g)}x^{(-g)} r = r^2cb
\end{equation}
 where
\[b=x_1^{(g_1)}\cdots x_p^{(g_p)} \cdot x_p^{(-g_p)}\cdots x_1^{(-g_1)}\]
which can be written as a polynomial in $u_i$, $i\le p$, and
\[c=x_{p+k+1}^{(g_{p+k+1})}\cdots x_n^{(g_n)} \cdot x_n^{(-g_n)}\cdots x_{p+k+1}^{(-g_{p+k+1})}\]
which can be written as a polynomial in $u_i$, $i>p$.
Since $b$ is regular in $A$, \eqref{eq:lem_step} implies $r^2c=0$.
We have the following isomorphisms of algebras
\[A\simeq A_{p|0}^-\otimes_\K A_{0|q}^-\simeq A_{p|0}^-\otimes_\K M_{2^q}(\K)\simeq M_{2^k}( A_{p|0}^-)\otimes_\K M_{2^{q-k}}(\K).\]
Under this isomorphism, $r^2c$ is mapped to $r^2\otimes c$. That this is zero implies $r^2=0$. But $R$ is isomorphic to $(\K[u_i\mid i\in\iv{1}{p}])^{2^q}$ which is a direct product of domains, hence $r=0$. This proves $a=0$.
\end{proof}

\subsection{Realization as TGW algebras} \label{sec:clifford-weyl-as-tgwa}

To realize $A_{p|q}^\pm$ as TGW algebras, consider the commutative $\K$-algebra
\begin{equation}\label{eq:Rpqdef}
R_{p|q}^\pm:=\K\big[u_i\mid i\in \iv{1}{n}\big]/\big( u_i^2-u_i\mid (-1)^{p(i)}=\pm 1\big).
\end{equation}
There is an injective homomorphism
\begin{equation} \label{eq:iotadef}
\begin{aligned}
\iota: R_{p|q}^\pm &\longrightarrow A_{p|q}^\pm,\\
u_i &\longmapsto \partial_ix_i.
\end{aligned}
\end{equation}
We will often use $\iota$ to identify $R_{p|q}^\pm$ with its image in $A_{p|q}^\pm$. One checks that
the image of $\iota$ coincides with the degree zero subalgebra, $(A_{p|q}^\pm)_0$, of $A_{p|q}^\pm$ with respect to the $\Z^n$-gradation
$A_{p|q}^\pm=\bigoplus_{d\in\Z^n} (A_{p|q}^\pm)_d$ given by $\deg(x_i)=\Be_i$,
$\deg(\partial_i)=-\Be_i,\;\forall i\in \iv{1}{n}$. 

For $i,j\in \iv{1}{n}$, put 
\begin{equation}\label{eq:laij}
\la_{ij}=\mp (-1)^{p(i)p(j)}
\end{equation}
and for $i\in \iv{1}{n}$, define $\tau_i\in\Aut_\K(R_{p|q}^\pm)$ by
\begin{equation}\label{eq:tau-definition}
\tau_i(u_j)=\begin{cases} \la_{ii}(u_i-1),&\text{if $i=j$,}\\ u_j,&\text{otherwise.}\end{cases}
\end{equation}
One checks that $\tau_i$ preserves the relations $u_j^2-u_j=0$ for $j$ with $(-1)^{p(j)}=\pm 1$.
Let $\tau=(\tau_i)_{i=1}^n$ and $u=(u_i)_{i=1}^n$.
Let $\TGWA{\la}{R_{p|q}^\pm}{\tau}{u}$ be the corresponding TGW algebra.

\begin{Theorem}\label{thm:WeylClifford}
There is an isomorphism of $\K$-algebras
\begin{equation}
\begin{aligned}
\chi:\TGWA{\la}{R_{p|q}^\pm}{\tau}{u} &\overset{\sim}{\longrightarrow}   A_{p|q}^\pm, \\
X_i &\longmapsto x_i,\\
Y_i &\longmapsto \partial_i.
\end{aligned}
\end{equation}
In particular, $\TGWA{\la}{R_I^\pm}{\tau}{u}$ is consistent (i.e. the natural map $\rho:R_{p|q}^\pm\to \TGWA{\la}{R_{p|q}^\pm}{\tau}{u}$  is injective).
\end{Theorem}

\begin{proof}
Put $R=R_{p|q}^\pm$ and $A=A_{p|q}^\pm$. The identities $x_i(\partial_ix_i)=(x_i\partial_i)x_i$, $x_i\partial_j=\la_{ij}\partial_jx_i$ for $i\neq j$,
and $x_i\partial_i=\la_{ii}(\partial_ix_i+1)$ imply that relations \eqref{eq:tgwarels} are preserved. 
Thus we have a map $\TGWC{\la}{R}{\tau}{u}\to A$ of $R$-rings given by $X_i\mapsto x_i$, $Y_i\mapsto \partial_i$. 
Furthermore, for each $i,j\in \iv{1}{n}$, it can be checked, using Lemma \ref{lem:monomials} that $[X_i,X_j]$ and $[Y_i,Y_j]$ lie in the radical of the gradation form on $\TGWC{\la}{R}{\tau}{u}$. In fact these elements generate the radical. Hence they generate the ideal $\TGWI{\la}{R}{\tau}{u}$ by Theorem \ref{thm:nondeg}. 
Since $[x_i,x_j]=[\partial_i,\partial_j]=0$ in $A$ this shows that we have a well-defined map $\chi:\TGWA{\la}{R}{\tau}{u}\to A$ of $R$-rings given by $X_i\mapsto x_i$, $Y_i\mapsto \partial_i$.
Since $x_i$ and $\partial_i$ generate $A$, the map $\chi$ is surjective. It remains to prove it is injective.
Since $\chi$ is a map of $R$-rings, the following diagram is commutative:
\begin{equation}\label{eq:WeylClifford-proof}
\begin{aligned}
\begin{tikzcd}[ampersand replacement=\&]
\TGWA{\la}{R}{\tau}{u} 
 \arrow{r}{\chi}  \& A \\ 
  R
 \arrow[hookrightarrow]{ur}[swap]{\iota}
 \arrow{u}{\rho}  \&
\end{tikzcd}
\end{aligned}
\end{equation}
Since $\iota$ is injective, $\rho$ is injective. That is, $\TGWA{\la}{R}{\tau}{u}$ is consistent. 
Identifying $R$ with the images under $\rho$ and $\iota$, the map $\chi|_R$ is the identity map.
Both $\TGWA{\la}{R}{\tau}{u}$ and $A$ are $\Z^n$-graded algebras and $\chi$ is a graded homomorphism. Therefore $J=\ker \chi$ is a graded ideal of $\TGWA{\la}{R}{\tau}{u}$.
If $J\neq 0$ then, since $\TGWA{\la}{R}{\tau}{u}$ is a consistent TGW algebra, $J\cap R\neq 0$.
However that contradicts that $\chi|_R$ is injective. Hence $J=0$ which completes the proof that $\chi$ is an isomorphism.
\end{proof}

\begin{Remark}
When $q>0$, Theorem \ref{thm:WeylClifford} implies that $A_{p|q}^-$ is a consistent TGW algebra which is not regularly graded. Indeed, if $j>p$, then $u_j$ is not regular in $R$ because $u_j(u_j+1)=0$. Thus, by Theorem \ref{thm:regularity}, $\TGWA{\la}{R}{\tau}{u}$ is not regularly graded. Note that for non-regularly graded TGW algebras, it is not known if relations \eqref{eq:consistency_rels} are sufficient (or even necessary) for it to be consistent.
\end{Remark}

\begin{Remark}
The algebra $A_{0|q}^-$ is finite-dimensional (an even Clifford algebra). Hence Theorem \ref{thm:WeylClifford} shows that TGW algebras can be finite-dimensional.
\end{Remark}

\begin{Remark}
Theorem \ref{thm:WeylClifford} suggests that the class of TGW algebras already contains not only quantum deformations of many algebras  (see e.g. \cite[Ex.~2.2.3]{MazTur2002}), but also supersymmetric analogues of certain algebras, without modifying the definition of TGW algebras.
\end{Remark}

\section{A new family of TGW algebras $\mathcal{A}(\ga)^\pm$}
\label{sec:family}

In this section we define a family of TGW algebras that depend on a matrix. This construction is a supersymmetric generalization of the one in \cite{HarSer2016}.

\subsection{Construction via monomial maps}

In this section we use the Clifford/Weyl superalgebras $A_{p|q}^\pm$ to construct new TGW algebras denoted $\mathcal{A}(\ga)^\pm$. Our method is to look for maps
\[\varphi:\mathcal{A}_\mu(R,\si,t)\to A_{p|q}^\pm \]
of $R$-rings with involution. Here $R=R_{p|q}^\pm$.
The motivation is at threefold. First it generalizes the construction from \cite{HarSer2016} which correseponds to the case $A_{p|0}^-$. Second, the TGW algebras obtained in this way automatically come with $\varphi$, which may be thought of as a representation by differential operators. Thirdly we show in Section \ref{sec:Lie-superalgebras} that certain quotients of enveloping algebras of Lie superalgebras are TGW algebras of exactly of this form.

As in \cite{HarSer2016} we restrict attention to monomial embeddings
\begin{equation}\label{eq:phi}
\varphi(X_i)=x_1^{(\ga_{1i})}x_2^{(\ga_{2i})}\cdots x_n^{(\ga_{ni})}.
\end{equation}
Here $n=p+q$, $\ga_{ji}\in\Z$ and we use the notation
\begin{equation}
x_j^{(k)} = \begin{cases}
 x_j^k,&k\ge 0,\\
 \partial_j^{-k}, &k<0.
\end{cases}
\end{equation}
In the case of \cite{HarSer2016} under mild assumptions on $\varphi$ the form \eqref{eq:phi} was in fact shown to be necessary. Here in our more general setting we shall be content with showing how the assumption that $\varphi$ is a homomorphism of $R$-rings with involution such that \eqref{eq:phi} holds naturally gives rise to conditions on $\ga_{ji}$ and also specifies the TGW datum (automorphisms $\si_i$, elements $t_i\in R$ and scalars $\mu_{ij}$).

First, since $\varphi$ is supposed to be a map of rings with involution, we necessarily have
\begin{equation}
\varphi(Y_i)=\varphi(X_i^\ast)=\varphi(X_i)^\ast=x_n^{(-\ga_{ni})}\cdots x_2^{(-\ga_{2i})} x_1^{(-\ga_{1i})}.
\end{equation}
Second, since $\varphi$ is a map of $R$-rings and $t_i\in R$ we have
\begin{multline}
t_i
=\varphi(t_i)
=\varphi(Y_iX_i)
=\varphi(Y_i)\varphi(X_i)
=x_n^{(-\ga_{ni})}\cdots x_2^{(-\ga_{2i})} x_1^{(-\ga_{1i})} \cdot x_1^{(\ga_{1i})}x_2^{(\ga_{2i})}\cdots x_n^{(\ga_{ni})}\\
=x_1^{(-\ga_{1i})}x_1^{(\ga_{1i})}\cdot x_2^{(-\ga_{2i})}x_2^{(\ga_{2i})}\cdots x_n^{(-\ga_{ni})}x_n^{(\ga_{ni})}
\end{multline}
In the last step we used the TGW algebra realization of $A_{p|q}^\pm$ which in particular has $\tau_i(u_j)=u_j$ for $j\neq i$.
To obtain an explicit formula for $t_i$ we compute $x_j^{(-\ga_{ji})}x_j^{(\ga_{ji})}$.
If $\ga_{ji}>0$ we have
\begin{equation}
x_j^{(-\ga_{ji})}x_j^{(\ga_{ji})}=\partial_j^{\ga_{ji}} x_j^{\ga_{ji}}=
\tau_j^{(-\ga_{ji}+1)}(u_j)\cdots \tau_j^{-1}(u_j)u_j
\end{equation}
Here we see that this is zero if $\la_{jj}=-1$ and $\ga_{ji}>1$ because then $\tau_j^{-1}(u_j)u_j = \tau_j^{-1}( u_j\la_{jj}(u_j-1))=0$ due to $u_j^2=u_j$ in $R$.
To avoid this scenario (having $t_i=0$ in a TGW algebra leads to degenerate behaviour such as $X_i=Y_i=0$) we make our first assumption on $\ga_{ji}$:
\begin{equation} \label{eq:gamma-assumption-1}
|\ga_{ji}|\le 1\quad\text{for all $i,j$ such that $\la_{jj}=-1$}
\end{equation}
Under this assumption we can proceed and obtain the formula
\begin{equation}
x_j^{(-\ga_{ji})}x_j^{(\ga_{ji})}=
(u_j+\ga_{ji}-1)\cdots(u_j+1)u_j
\end{equation}
We used that $\tau_j(u_j)=\la_{jj}(u_j-1)$, so the formula is clear when $\la_{jj}=1$ while if $\la_{jj}=-1$ there is at most one factor (empty product is interpreted as $1$.)
The case $\ga_{ji}<0$ is handled analogously (which is why we put absolute value in \eqref{eq:gamma-assumption-1}).

The final formula for $t_i$ is
\begin{equation}\label{eq:ti-formula}
\begin{gathered}
t_i=u_{1i}u_{2i}\cdots u_{ni} \\
u_{ji} = \begin{cases}
 (u_j+\ga_{ji}-1)\cdots (u_j+1)u_j,& \ga_{ji}>0,\\
 1,&\ga_{ji}=0,\\
 (u_j-|\ga_{ji}|)\cdots (u_j-2)(u_j-1),&\ga_{ji}<0.
 \end{cases}
\end{gathered}
\end{equation}
Similarly $\si_i$ can be deduced as follows. We have
\[\varphi(X_iu_j)=\varphi(\si_i(u_j)X_i).\]
Since $\varphi$ is a homomorphism of $R$-rings, we have
\[\varphi(X_i)u_j=\si_i(u_j) \varphi(X_i).\]
Substituting \eqref{eq:phi} we immediately obtain the sufficient condition
\begin{equation}\label{eq:sigma-from-gamma}
\si_i = \tau_1^{\ga_{1i}}\tau_2^{\ga_{2i}}\cdots\tau_n^{\ga_{ni}}
\end{equation}
What remains is to ensure that for $i\neq j$,
\[X_iY_j = \mu_{ij} Y_jX_i\]
holds for appropriate scalars $\mu_{ij}$, under suitable assumptions on $\ga_{kl}$.
We have
\[\varphi(X_i)\varphi(Y_j)=x_1^{(\ga_{1i})}\cdots x_n^{(\ga_{ni})}\cdot x_n^{(-\ga_{nj})}\cdots x_1^{(-\ga_{1j})}\]
First we observe that if there exists $k\in\{1,2,\ldots,n\}$ with $\la_{kk}=-1$ and $\ga_{ki}\ga_{kj}<0$ then $\varphi(X_i)\varphi(Y_j)=0=\varphi(Y_j)\varphi(X_i)$. 
If no such $k$ exists we want to move all factors on the right $x_l^{(-\ga_{lj})}$ to the left of all factors $x_k^{(\ga_{ki})}$. The only problem is when $k=l$. A natural assumption for it to be possible is that actually $\ga_{ki}\ga_{kj}\le 0$, because then the two factors are either both powers of $x_k$ or both powers of $\partial_k$.

To summarize, we make the following second assumption on $\ga_{ji}$:
\begin{equation}
\forall i\neq j: \text{Either $\ga_{ki}\ga_{kj}<0$ for some $k$ with $\la_{kk}=-1$, or $\ga_{ki}\ga_{kj}\le 0$ for all $k$.}
\end{equation}
Under this assumption we then have for all $k,l$:
\begin{equation}
x_k^{(\ga_{ki})}x_l^{(\ga_{lj})} = \la_{kl}^{\ga_{ki}\ga_{lj}} x_l^{(\ga_{lj})} x_k^{(\ga_{ki})}.
\end{equation}
Thus we finally obtain that
\[\varphi(X_i)\varphi(Y_j)=\mu_{ij}\varphi(Y_j)\varphi(X_i)\]
holds, provided
\begin{equation}
\mu_{ij} = \prod_{1\le k,l\le n} \la_{kl}^{\ga_{ki}\ga_{lj}}
\end{equation}
Using that $\la_{kl}=(\mp 1)(-1)^{p(k)p(l)}$ this can be written
\begin{equation}\label{eq:mu-from-gamma}
\mu_{ij}=\mu_{ij}^\pm = (\mp 1)^{p'(i)p'(j)} \cdot (-1)^{p(i)p(j)}
\end{equation}
where the parities are defined by
\begin{align}
\label{eq:p-parity}
p(i) &=\sum_{k=1}^n \bar \ga_{ki} p(k) \\
\label{eq:p-prime-parity}
p'(i)&=\sum_{k=1}^n \bar \ga_{ki}
\end{align}
($\bar x\in \Z/2\Z$ is the image of $x\in \Z$ under the canonical projection.)

Note that \eqref{eq:p-parity} expresses that the matrix $\ga$, when regarded as a $\Z$-module map $\Z^m\to \Z^n$, is an even map, with respect to the parity $p(a_1,\ldots,a_n)=\sum_k \bar a_k p(k)$.

\begin{Theorem}\label{thm:sufficient}
Let $p,q,m$ be non-negative integers, put $n=p+q$. 
Let $\ga=(\ga_{ji})$ be a $n\times m$-matrix with integer entries satisfying the following two conditions:
\begin{enumerate}[{\rm (i)}]
\item $|\ga_{ji}|\le 1$ whenever $\la_{ii}=-1$,
\item $\forall i\neq j$: either $\ga_{ki}\ga_{kj}<0$ for some $k$ with $\la_{kk}=-1$, or $\ga_{ki}\ga_{kj}\le 0$ for all $k$.
\end{enumerate}
Then there exist a TGW algebra $\mathcal{A}(\ga)^\pm=\mathcal{A}_{\mu}(R,\si,t)$ of rank $m$, and a homomorphism of $R$-rings with involution 
\begin{equation} \label{eq:varphi-definition-pm}
\varphi:\mathcal{A}_{\mu}(R,\si,t)\to A_{p|q}^\pm.
\end{equation}
The homomorphism is uniquely determined by the condition
\begin{equation}
\varphi(X_i)=x_1^{(\ga_{1i})}x_2^{(\ga_{2i})}\cdots x_n^{(\ga_{ni})}
\end{equation}
and the TGW algebra is given by the following data:
\begin{equation} \label{eq:Rpq-explicit-definition}
R=R_{p|q}^\pm=\K[u_1,\ldots,u_n]/(u_i^2-u_i\mid \la_{ii}=-1)
\end{equation}
where $\la_{ij}=\mp (-1)^{p(i)p(j)}$ and $t=(t_1,\ldots,t_m)$ where
\begin{equation}\label{eq:ti-formula2}
\begin{gathered}
t_i=u_{1i}u_{2i}\cdots u_{ni} \\
u_{ji} = \begin{cases}
 (u_j+\ga_{ji}-1)\cdots (u_j+1)u_j,& \ga_{ji}>0,\\
 1,&\ga_{ji}=0,\\
 (u_j-|\ga_{ji}|)\cdots (u_j-2)(u_j-1),&\ga_{ji}<0.
 \end{cases}
\end{gathered}
\end{equation}
Lastly, $\si=(\si_1,\ldots,\si_m)$ where
\begin{equation}\label{eq:sigma-from-gamma2}
\si_i = \tau_1^{\ga_{1i}}\tau_2^{\ga_{2i}}\cdots\tau_n^{\ga_{ni}}
\end{equation}
where
\begin{equation}\label{eq:tau-definition-again}
\tau_i(u_j)=\begin{cases} \la_{ii}(u_i-1),&\text{if $i=j$,}\\ u_j,&\text{otherwise.}\end{cases}
\end{equation}
and $\mu=(\mu_{ij})_{1\le i,j\le m}$ where
\begin{equation}
\mu_{ij}=(\mp 1)^{p'(i)p'(j)} \cdot (-1)^{p(i)p(j)}
\end{equation}
where $p(i)$ and $p'(i)$ were defined in \eqref{eq:p-parity}-\eqref{eq:p-prime-parity}.
\end{Theorem}

\begin{proof}
The discussion preceeding the theorem proves that there exists a homomorphism of $R$-rings with involution
\begin{equation}
\varphi':\mathcal{C}=\mathcal{C}_\mu(R,\si,t)\to A_{p|q}^\pm 
\end{equation}
All that remains is to show that $\varphi'(\mathcal{I})=0$ where $\mathcal{I}$ is the unique maximal $\Z^m$-graded ideal trivially intersecting the degree zero component of $\mathcal{C}$.
If $a$ is a homogeneous element of $\mathcal{I}$ then $a^\ast\cdot a=0$ hence $\varphi'(a)^\ast \cdot \varphi'(a)=0$. By Lemma \ref{lem:aaast1}, it follows that $\varphi(a)=0$. 
\end{proof}

\begin{Remark}
Theorem \ref{thm:sufficient} provides a large family of consistent non-regular TGW algebras.
\end{Remark}

\begin{Remark} Let $p=3$, $q=2$, $m=4$ and
\[
\ga=\begin{bmatrix}
1  &    &      &   \\
-1 & 1  &      &   \\
   & -1 &  1    &    \\
\hdashline[2pt/2pt]
   &    & -1   & 1 \\
   &    &      &-1 \\ 
\end{bmatrix}
\]
(The dashed line separates even from odd rows.)
The corresponding TGW algebra $\mathcal{A}(\ga)^-$ is a quotient of $U\big(\Fgl(3|2)\big)$ (see Section \ref{sec:Lie-superalgebras}).
\end{Remark}

\subsection{Injectivity of $\varphi$}\label{sec:inj}
We prove a theorem which gives equivalent conditions for $\varphi$ defined in \eqref{eq:varphi-definition-pm} to be injective. This result will be used in Section \ref{sec:Lie-superalgebras}.

\begin{Lemma}[Weak injectivity of $\varphi$]
\label{lem:weak-inj}
If $g\in \Z^m$ and $a\in\mathcal{A}(\ga)^\pm_g$, $a\neq 0$, then $\varphi(a)\neq 0$.
\end{Lemma}

\begin{proof}
Suppose $a\neq 0$. Then, by the non-degeneracy of the gradation form of a TGW algebra, $ba\neq 0$ for some $b\in\mathcal{A}(\ga)^\pm_{-g}$. Applying $\varphi$ we get $\varphi(ba)\neq 0$ since $\varphi|_{R_E}$ is injective. Hence $\varphi(b)\varphi(a)\neq 0$, so in particular $\varphi(a)\neq 0$. 
\end{proof}

Let $\ast:A_{p|q}^\pm\to A_{p|q}^\pm$, $a\mapsto a^\ast$, be the unique $\K$-linear map satisfying
$(a^\ast)^\ast=a,\,(ab)^\ast=b^\ast a^\ast$ for all $a,b\in A_{p|q}^\pm$, and $x_i^\ast=\partial_i$ for all $i$.

\begin{Lemma}\label{lem:aaast}
Let $\ga$ be a matrix satisfying the conditions of Theorem \ref{thm:sufficient} and let $\mathcal{A}(\ga)^\pm$ be the corresponding TGW algebra. Let $a\in\mathcal{A}(\ga)^\pm$ be a homogeneous element of degree $g\in \Z^m$. If $a^\ast \cdot a=0$ then $a=0$.
\end{Lemma}

\begin{proof}
Suppose $a\neq 0$. By Lemma \ref{lem:weak-inj}, $\varphi(a)\neq 0$. So, by Lemma \ref{lem:aaast1}, $\varphi(a)^\ast\cdot\varphi(a)\neq 0$. Since $\varphi$ is a map of rings with involution, $\varphi(a^\ast\cdot a)\neq 0$. Hence $a^\ast \cdot a\neq 0$.
\end{proof}

\begin{Remark}
If $\la_{ii}=1$ for all $i$ then $R_{p|q}^\pm$ 
defined in \eqref{eq:Rpq-explicit-definition} is a domain. Then, by \cite[Prop.~2.9]{FutHar2012b}, $\mathcal{A}(\ga)^\pm$ is also a domain. Hence Lemma \ref{lem:aaast} holds trivially in this case.
\end{Remark}

For a $\Z I$-graded algebra $A=\oplus_{g\in\Z I}A_g$ we define the \emph{(graded) support} of $A$ to be $\Supp(A):=\big\{g\in\Z I\mid A_g\neq \{0\}\big\}$.

\begin{Lemma}\label{lem:supersupport}
Let $\mathcal{A}(\ga)^\pm$ be a TGW algebra as constructed in Theorem \ref{thm:sufficient}. 
Let $S^\pm\subseteq \Z^m$ be the support of $\mathcal{A}(\ga)^\pm$.
Then, regarding $\ga$ as a $\Z$-linear map from $\Z^m$ to $\Z^n$ we have
\begin{equation}
\begin{aligned}
\ga(S^+) &\subseteq \{-1,0,1\}^p \times \Z^q,\\
\ga(S^-) &\subseteq \Z^p\times \{-1,0,1\}^q.
\end{aligned}
\end{equation}
\end{Lemma}

\begin{proof}
We consider the case $S^-$. The other case is analogous.
Let $g\in S^-$. Since any TGW algebra is generated as a left $R$ module by the reduced monomials (Lemma \ref{lem:monomials}), there exist sequences
$(i_1,i_2,\ldots,i_k)$ and $(j_1,j_2,\ldots,j_l)$ of elements from $\{1,2,\ldots,m\}$ with $\{i_1,i_2,\ldots,i_k\}\cap\{j_1,j_2,\ldots,j_l\}=\emptyset$ such that
\[a=Y_{i_1}Y_{i_2}\cdots Y_{i_k}\cdot X_{j_1}X_{j_2}\cdots X_{j_l}\]
is a nonzero element in $\mathcal{A}(\ga)^-_g$. By Lemma \ref{lem:weak-inj}, $\varphi(a)\neq 0$. We have
\[\varphi(a)=\prod_{r=1}^m
x_r^{(-\ga_{ri_1})}\cdots x_r^{(-\ga_{ri_k})}\cdot x_r^{(\ga_{rj_1})}\cdots x_r^{(\ga_{rj_l})}.
\]
For $r>p$, a product of the form
\[x_r^{(-\ga_{ri_1})}\cdots x_r^{(-\ga_{ri_k})}\cdot x_r^{(\ga_{rj_1})}\cdots x_r^{(\ga_{rj_l})} \]
can only be nonzero if the factors $x_r^{(\beta)}$ alternate between $x_r$ and $\partial_r$ (ignoring factors where $\beta=0$). In particular, the number of $x_r$'s must differ from the number of $\partial_r$'s by at most one.
\end{proof}

To prove that homomorphisms from TGW algebras are injective, the following result is useful.

\begin{Theorem}[{\cite[Thm.~3.6]{HarOin2013}}]
\label{thm:essential}
If $A=\TGWA{\mu}{R}{\si}{t}$ is consistent, then the centralizer $C_A(R)$ of $R$ in $A$ is an essential subalgebra of $A$, in the sense that $J\cap C_A(R)\neq \{0\}$ for any nonzero ideal $J$ of $A$.
\end{Theorem}

\begin{Theorem}\label{thm:super_injectivity}
Let $\ga$ be a matrix as in Theorem \ref{thm:sufficient} and $A=\mathcal{A}(\ga)^\pm$ be the corresponding TGW algebra. Put $R=R_{p|q}^\pm$. The following statements are equivalent.
\begin{enumerate}[{\rm (i)}]
\item $R$ is a maximal commutative subalgebra of $A$;
\item If $g\in\Supp(A)$ is such that $\si_g:=\prod_{i=1}^m \si_i^{g_i}=\Id_R$, then $g=0$;
\item Put
\begin{equation}
\Z^{p|q}_- =\Z^p\times (\Z/2\Z)^q,\qquad 
\Z^{p|q}_+ = (\Z/2\Z)^p\times \Z^q.
\end{equation}
Then the composition
\begin{equation}
\Supp(A)\longrightarrow \Z^m \overset{\ga}{\longrightarrow} \Z^n= \Z^p\times \Z^q \overset{P}{\longrightarrow} \Z^{p|q}_\pm 
\end{equation}
is injective (the first map is inclusion and the last is canonical projection);
\item The restriction of $\ga:\Z^m\to \Z^n$ to $\Supp(A)$ is injective;
\item The map $\varphi$ defined in \eqref{eq:varphi-definition-pm} is injective.
\end{enumerate}
\end{Theorem}

\begin{proof}
(i)$\Rightarrow$(ii): Suppose $g\in\Supp(A)$ with $\si_g=\Id_R$. Then for any $a\in A_g$ and $r\in R$ we have $ar=\si_g(r)a=ra$ which means that $A_g\subseteq C_A(R)$. But $C_A(R)=R$ by (i). Thus, since $A_g\neq \{0\}$, this means that $g$ must be $0$ and $A_g=R$.

\noindent(ii)$\Rightarrow$(iii): Suppose $P\circ \gamma(g)=0$ in $\Z^{p|q}_\pm$ for some $g\in\Supp(A)$. Then $\si_g=\prod_{r=1}^n\tau_r^{\ga(g)_r}=\Id_R$ because $\tau_r^2=\Id_R$ for $r>p$
when $\pm=-$ and for $r\le p$ when $\pm=+$. By (ii) this implies $g=0$.

\noindent(iii)$\Rightarrow$(i): For simplicity we assume $\pm=-$. The other case is symmetric.
Suppose $a\in C_A(R)$, $a\neq 0$. Since $C_A(R)$ is a graded subalgebra of $A$ we may without loss of generality suppose there exists $g\in\Z^m$ such that $a\in A_g\cap C_A(R)$. Since $a\neq 0$, this implies $g\in\Supp(A)$. For all $r\in R$ we have $(\si_g(r)-r)a=ar-ra=0$. Taking $r=u_j$ we get
\[0=(\si_g(u_j)-u_j)a=(\tau_j^{\ga(g)_j}(u_j)-u_j)a=\begin{cases}
-\ga(g)_ja,& j\le p,\\
0,& j>p, \ga(g)_j=0 \text{ in $\Z/2\Z$}\\
(1-2u_j)a,& j>p, \ga(g)_j=1+2\Z.
\end{cases}\]
Since $a\neq 0$, we get $\ga(g)_j=0$ for all $j\le p$.
Suppose $j>p$ and $\ga(g)_j=1+2\Z$, then $0=u_j(1-2u_j)a=-u_ja$ since $u_j^2=u_j$. Combining this with $(1-2u_j)a=0$ we get $a=0$, a contradiction.
Therefore, for $j>p$ we must have $\ga(g)_j=0$ in $\Z/2\Z$. This proves that $\ga(g)=0$ in $\Z^p\oplus (\Z/2\Z)^q$.

\noindent(iii)$\Rightarrow$ (iv): Trivial.

\noindent(iv)$\Rightarrow$ (iii): Suppose $P\circ\ga(g)=0$ for some $g\in\Supp(A)$. By Lemma \ref{lem:supersupport} we get $\gamma(g)=0$ so by (iv), $g=0$.

\noindent(i)$\Rightarrow$(v): Let $K=\ker(\varphi)$. If $K\neq\{0\}$, then by Theorem \ref{thm:essential}, $K\cap C_A(R)\neq \{0\}$. By (i), $C_A(R)=R$. Hence $K\cap R\neq \{0\}$. But by Theorem \ref{thm:sufficient}, $\varphi$ is a map of $R$-rings with involution and thus in particular $\varphi|_R=\Id_R$ (where we used the injective maps $\rho$ and $\iota$ to identify $R$ with its image in $A$ and $A_E(\K)$ respectively). This contradiction shows that $K=\{0\}$.

\noindent(v)$\Rightarrow$(i): If $a\in C_A(R)$ then $\varphi(a)\in C_{A_{p|q}^\pm}(R)$ which equals $R$ by Lemma \ref{lem:RpqMaxComm}. By (v) this implies $a\in R$.
\end{proof}

\begin{Example} \label{ex:inj-g}
 Let $p,q$ be non-negative integers and $n=p+q>0$. Consider the matrices
\[
\al=\begin{bmatrix}
1  &    &        &  \\
-1 & 1  &        & \\
   & -1 &        &  \\
   &    & \ddots & 1 \\
   &    &        & -1
\end{bmatrix}
\quad
\be=\begin{bmatrix}
1  &    &        & & \\
-1 & 1  &        & &\\
   & -1 &        &  &\\
   &    & \ddots & 1 &\\
   &    &        & -1 & 1
\end{bmatrix}
\quad 
\ga=\begin{bmatrix}
1  &    &        & & \\
-1 & 1  &        & &\\
   & -1 &        &  &\\
   &    & \ddots & 1 & \\
   &    &        & -1 & 2 
\end{bmatrix}
\]
These are $n\times m$ matrices (where $m=n-1$ in the case of $\al$ and $m=n$ for $\be,\ga$) and define $\Z$-linear maps $\Z^m\to \Z^n$. In each case the top $p$ rows are defined to be even and the remaining $q$ rows are odd. 
It is easy to see that these maps are injective, hence by Theorem \ref{thm:super_injectivity}(iv)$\Rightarrow$(v), the homomorphism $\varphi:\mathcal{A}(\zeta)^\pm\to A_{p|q}^\pm$ is injective for $\zeta=\al,\be,\ga$.
\end{Example}

\subsection{A description of the graded support of $\mathcal{A}(\ga)^-$}
\label{sec:support}
Although sufficient for the application to Lie superalgebras, the characterization in Theorem \ref{thm:super_injectivity} of the injectivity of the map \eqref{eq:varphi-definition-pm} is not completely satisfactory because we lack a good description of the support of $\mathcal{A}(\ga)^\pm$.
In this section we give a combinatorial description of the support of $\mathcal{A}(\ga)^-$ in terms of certain pattern-avoiding vector compositions of the columns of $\ga$. A similar analysis applies to $\mathcal{A}(\ga)^+$. This allows us to compute the support in the certain cases. In addition, it shows that that this is a non-trivial problem for a general (non-regular) TGW algebra.

Put $W=\Z^d$. A \emph{$d$-dimensional vector composition} of $w\in W$ is a tuple $c=(c_1,c_2,\ldots,c_\ell)\in W^\ell$ such that $c_1+c_2+\cdots+c_\ell=w$. The non-negative integer $\ell$ is the \emph{length} of $c$. The $c_i$ are called the \emph{parts} of the composition $c$. A given vector $u\in W$ \emph{appears with multiplicity $m$ (in $c$)} if $c_j=u$ for exactly $m$ choices of $j\in\iv{1}{\ell}$.
\begin{Example} \label{ex:comp}
$
\left(
\left[\begin{smallmatrix}1\\  1\\ 1\end{smallmatrix}\right],
\left[\begin{smallmatrix}3\\  0\\-1\end{smallmatrix}\right]
\left[\begin{smallmatrix}0\\ -1\\ 1\end{smallmatrix}\right],
\left[\begin{smallmatrix}3\\  0\\-1\end{smallmatrix}\right]
\right)$
is a $3$-dimensional vector composition of 
$\left[\begin{smallmatrix}7\\0\\0\end{smallmatrix}\right]$.
\end{Example}

\begin{Theorem}
Let $A=\mathcal{A}(\ga)^-$ be a TGW algebra constructed as in Theorem \ref{thm:sufficient}. The following are equivalent for $g\in \Z^m$:
\begin{enumerate}[{\rm (i)}]
\item $g\in \Supp(A)$;
\item there exists an $n$-dimensional vector composition of $\ga(g)$ of length $|g|=\sum_{i\in V}|g_i|$ such that 
\begin{enumerate}[{\rm (a)}]
\item each part is of the form $\sgn(g_i)\ga(\Be_i)$ for $i\in V$ which appears with multiplicity $|g_i|$,
\item for each $r>p$ the sequence $\big( \sgn(g_{i_1})\ga_{ri_1},\ldots,\sgn(g_{i_{|g|}})\ga_{ri_{|g|}}\big)$ contains no consecutive subsequence of the form
\[(1,0,\ldots,0,1)\quad\text{or}\quad (-1,0,\ldots,0,-1)\] 
where there are zero or more $0$'s.
\end{enumerate}
\end{enumerate}
\end{Theorem}

\begin{proof}
By Lemma \ref{lem:monomials}, $g\in\Supp(A)$ if and only if $A_g$ contains a reduced monomial $a=Z_{i_1}Z_{i_2}\cdots Z_{i_{|g|}}$ (where each $Z_{i_k}\in\cup_{j\in V} \{X_j,Y_j\}$) such that $a\neq 0$, which by
Lemma \ref{lem:weak-inj} is equivalent to $\varphi(a)\neq 0$.
Put $\ep_k=\sgn(g_{i_k})$. We have
\[\varphi(a)= \varphi(Z_{i_1})\cdots\varphi(Z_{i_{|g|}})= \pm \prod_{r\in E} x_r^{(\ep_1 \ga_{ri_1})} \cdots x_r^{(\ep_{|g|} \ga_{ri_{|g|}})} \]
which is nonzero if and only if property (b) in the theorem holds.
\end{proof}

\begin{Example}
If $q=0$ then $\Supp\big(\mathcal{A}(\ga)^-\big)=\Z^m$ because condition (b) is void.
\end{Example}

\begin{Example}
Let $m=3, p=1, q=2$ and 
$\ga=\begin{bmatrix} 1& 3& 0 \\ \hdashline[2pt/2pt]  1 & 0 & -1\\ 1 & -1 & 1 \end{bmatrix}$.
Example \ref{ex:comp} shows that $(1,2,1)$ belongs to the graded support of the TGW algebra $\mathcal{A}(\ga)^-$.
On the other hand $(2,1,0)$ does not, because there is no vector composition of length $3$ with two parts equal to
$\left[\begin{smallmatrix}1\\  1\\ 1\end{smallmatrix}\right]$ and one part equal to
$\left[\begin{smallmatrix}3\\  0\\ -1\end{smallmatrix}\right]$ which avoids the pattern $(1,0,\ldots,0,1)$ in the second row.
\end{Example}

\begin{Example}
Let $m=2, p=0, q=1$ and $\ga=\begin{bmatrix} 1& -1\end{bmatrix}$. Then
\[\Supp\big(\mathcal{A}(\ga)\big)=\{(g_1,g_2)\in\Z^2 \mid  |g_1-g_2|\leq 1\}.\]
\end{Example}

\begin{Example}
Let $m=2, p=0, q=2$, $\ga=\begin{bmatrix} 1 & 0\\ 1 & -1 \end{bmatrix}$, then 
\[\Supp\big(\mathcal{A}(\ga)\big)=\{(0,0),\pm (0,1), \pm (1,0), \pm (1,1), \pm (1,2)\}.\]
\end{Example}

\section{Relation to $\mathfrak{gl}(m|n)$ and $\mathfrak{osp}(m|2n)$}
\label{sec:Lie-superalgebras}

Irreducible completely pointed weight modules have been classified and realized by ($q$-)differential operators in the case of simple finite-dimensional complex Lie algebras $\Fg$ in \cite{BriLem1987,BenBriLem1997} and over $U_q(\Fsl_n)$ in \cite{FutHarWil2015}.
In \cite[Sec.~6]{Cou2013}, Coulembier classified all irreducible completely pointed highest weight modules over the orthosymplectic Lie superalgebras $\Fosp(m|2n)$, and realized them by differential operators on supersymmetric Grassmann algebras. See also \cite{Nis1990} for a uniform treatment of spinor representations of orthosymplectic Lie superalgebras.
In this section we show that, analogously to the Lie algebra case \cite{HarSer2016}, the realization of $\Fosp(m|2n)$ by differential operators factors through a corresponding twisted generalized Weyl algebra of the form $\mathcal{A}(\al)$.

Recall that the Lie superalgebra $\mathfrak{gl}(m|n)$ is the Lie superalgebra of all linear transformations of $(m|n)$-dimensional vector superspace, and
$\mathfrak{osp}(m|2n)$ is the subalgebra of  $\mathfrak{gl}(m|2n)$ preserving a non-degenerate even symmetric bilinear form on an $(m|2n)$-dimensional vector 
superspace or, equivalently, the subalgebra of $\mathfrak{gl}(2n|m)$ preserving a non-degenerate even skew-symmetric bilinear form on an $(2n|m)$-dimensional 
vector  superspace.
The even part of $\mathfrak{osp}(m|2n)$ is the direct sum $\mathfrak{so}(m)\oplus\mathfrak{sp}(2n)$.
The Lie superalgebras $\mathfrak{gl}(m|n)$ and $\mathfrak{osp}(m|2n)$ are Kac-Moody superalgebras and can be described by Chevally generators and relations,
see \cite{Kac77}, as follows.
Let $p, q$ be nonnegative integers, $n=p+q>0$.
The Chevalley generators of $\mathfrak{gl}(p|q)$ are $e_1,\dots,e_{n-1}$, $h_1,\dots,h_n$, $f_1,\dots,f_{n-1}$, with convention that $e_p,f_p$ 
are odd and all other generators are even. They satisfy the relations
\begin{gather*}
[h_i,h_j]=0,\quad  [h_i,e_j]=\delta_{i,j}e_j-\delta_{i,j+1}e_j,\quad  [h_i,f_j]=-\delta_{i,j}f_j+\delta_{i,j+1}f_j, \\ [e_i,f_j]=\delta_{i,j}(h_i-(-1)^{\delta_{ip}}h_{i+1}).
\end{gather*}
The Lie superalgebra $\mathfrak{gl}(p|q)$ is the quotient of the infinite-dimensional Lie algebra with the above relations by the maximal ideal 
which intersects trivially the Cartan subalgebra generated by $h_1,\dots,h_n$.
The Chevalley generators of $\mathfrak{osp}(2p+1|2q)$ are obtained from those for $\mathfrak{gl}(p|q)$ by adding odd generators  $e_n,f_n$ and relations
$$[h_i,e_n]=\delta_{i,n}e_n,\quad [h_i,f_n]=-\delta_{i,n}f_n, \quad [e_n,f_n]=h_n, \quad [e_i,f_n]=[e_n,f_i]=0\;\text{if $n\neq i$}.$$
The Chevalley generators of $\mathfrak{osp}(2p|2q)$ are obtained from those for $\mathfrak{gl}(p|q)$ by adding even generators  $e^2_n,f^2_n$.
From the above description it is not difficult to see that we have an embeddings of Lie superalgebras 
$$\mathfrak{gl}(p|q)\subset \mathfrak{osp}(2p|2q)\subset \mathfrak{osp}(2p+1|2q).$$

\subsection{Weyl superalgebra and $\mathfrak{osp}(2p|2q)$}
Let $V$ be a vector superspace equipped with even skew-symmetric form $\omega:V\times V\to \K$. We define the Weyl superalgebra $W(V,\omega)$
as the quotient of the tensor superalgebra
$T(V)$ by the relations
$$v\otimes w-(-1)^{p(v)p(w)}w\otimes v=\omega(v,w).$$

\begin{Lemma}\label{Weyl} Let $\Fg$ denote the span of the elements of the form $vw+(-1)^{p(v)p(w)}wv$ for all $v,w\in V$. Then $\Fg$ is closed under the 
supercommutator
and the adjoint action of $\Fg$ on $V$ preserves the form $\omega$.
\end{Lemma}
\begin{proof} Note that
$$vw+(-1)^{p(v)p(w)}wv=2vw-\omega(v,w)$$
and
$$[vw,u]=v[w,u]+(-1)^{p(w)p(u)}[v,u]w=\omega(w,u)v+(-1)^{p(w)p(u)}\omega(v,u)w.$$
The super Jacobi identity ensures that $\omega$ is $\operatorname{ad}_{vw}$-invariant. Indeed,
\begin{multline*}
\omega([vw,u_1],u_2)+(-1)^{p(vw)p(u_1)}\omega(u_1,[vw,u_2])  \\
=[[vw,u_1],u_2]+(-1)^{p(vw)p(u_1)}[u_1,[vw,u_2]] =[vw,[u_1,u_2]]=0.
\end{multline*}
Finally, $\Fg$ is closed under supercommutator as
$$[vw,xz]=[vw,x]z+(-1)^{p(vw)p(x)}x[vw,z]=[vw,x]z+(-1)^{p(vw)p(xz)}[vw,z]x.$$
\end{proof}

\begin{Corollary}\label{corWeyl} If $\omega$ is non-degenerate then $\Fg$ constructed in the previous lemma is isomorphic to
$\mathfrak{osp}(r|s)$ where $r=\dim V_1$ and $s=\dim V_0$.
\end{Corollary}

Let us assume that the $\omega$ is non-degenerate and both $r$ and $s$ are even. Set $r=2p$, $s=2q$ and $n=p+q$. Choose  basis $x_1,\dots,x_n,y_1,\dots,y_n$
in $V$ such that
$$\omega(x_i,x_j)=\omega(y_i,y_j)=0,\quad \omega(y_i,x_j)=\delta_{i,j}.$$
The parity is defined by 
$$p(x_i)=p(y_i)=\begin{cases} 1\,\,\text{if}\,\,i\leq p\\ 0\,\,\text{if}\,\,i> p\end{cases}.$$
In this case the Weyl algebra is isomorphic to $A^-_{q|p}$ since the defining relations are
$$x_ix_j-(-1)^{p(i)p(j)}x_jx_i=y_iy_j-(-1)^{p(i)p(j)}y_jy_i=0,$$
$$y_ix_j-(-1)^{p(i)p(j)}x_jy_i=\delta_{ij}.$$

Let $\Fg=\mathfrak{gl}(p|q)$, or $\mathfrak{osp}(2p|2q)$ and identify $\Z^m$ with the root lattice of $\Fg$ with basis 
consisting of the distinguished simple roots of $\Fg$. Let $\zeta:\Z^m \to \Z^n$ be the $\Z$-linear maps given 
by the matrices
\[
\begin{bmatrix}
1  &    &        &  \\
-1 & 1  &        & \\
   & -1 &        &  \\
   &    & \ddots & 1 \\
   &    &        & -1
\end{bmatrix},
\qquad
\begin{bmatrix}
1  &    &        & & \\
-1 & 1  &        & &\\
   & -1 &        &  &\\
   &    & \ddots & 1 &\\
   &    &        & -1 & 2
\end{bmatrix}
\]
respectively. Let $A^-_{q|p}$ be the Weyl superalgebra.

\begin{Theorem}\label{thmWeyl}
Let $p, q$ be nonnegative integers, $n=p+q>0$. Let $\Fg=\mathfrak{gl}(p|q)$, or $\mathfrak{osp}(2p|2q)$ and let $\zeta$ be as above. Then there is a 
commutative triangle of associative algebras with involution
\begin{equation} \label{eq:super-triangle}
\begin{tikzcd}[ampersand replacement=\&,
               column sep=small]
U\big(\mathfrak{g}\big) 
 \arrow[rr,"\pi"] 
 \arrow[rd,"\psi"'] \& \& A^-_{q|p} \\ 
\&  \mathcal{A}(\zeta)^-
 \ar[hookrightarrow,ur,"\varphi"'] \&
\end{tikzcd}
\end{equation}
where
$\varphi$ is given by Theorem \ref{thm:sufficient}, 
$\psi(e_i)=X_i$, $\psi(f_i)=Y_i$, $\psi(h_{ii})=\la_{ii}(u_i-1)$,
and
\begin{equation}
\pi(e_i)=
\begin{cases}
x_{i}\partial_{i+1}, & i<n,\\
x^2_n, & i=n,
\end{cases}\qquad
\pi(f_i)=\pi(e_i)^\ast,\;\pi(h_i)=x_i\partial_i+(-1)^{p(i)}\frac{1}{2},
\end{equation}
\end{Theorem}

\begin{proof}
First, the existence of $\pi$ follows from Corollary \ref{corWeyl}.  We need to check that $\tilde\pi(\mathfrak j)=0$. This follows immediately from the 
fact that  
$\tilde\pi(\mathfrak h)$ is the self-centralizing subalgebra of $\tilde\pi(\tilde{\mathfrak g})$.
Therefore we have a map $\tilde\pi:\mathfrak g\to A^-_{q|p}$ which extends to the homomorphism $\pi: U(\mathfrak g)\to A^-_{q|p}$ of associative algebras.
By Theorem \ref{thm:super_injectivity}, $\varphi$ is injective. Moreover, the image of $\varphi$  
coincides with the image of $\pi$. This immediately proves the existence of a unique map $\psi$ such that the diagram commutes.
\end{proof}

\subsection{Clifford superalgebra and $\mathfrak{osp}(2p+1|2q)$}
Let $V$ be a vector superspace equipped with even symmetric form $\beta:V\times V\to \K$. We define the Clifford superalgebra $\Cliff(V,\beta)$
as the quotient of the tensor superalgebra
$T(V)$ by the relations
$$v\otimes w+(-1)^{p(v)p(w)}w\otimes v=\beta(v,w).$$

Note that $\Cliff(V,\beta)$ is finite-dimensional iff $V$ is purely even. As any associative superalgebra  $\Cliff(V,\beta)$ has 
the associated Lie superalgebra
structure defined by $[x,y]=xy-(-1)^{p(x)p(y)}yx$.
Let $\Fg$ denote the Lie subalgebra of $\Cliff(V,\beta)$ generated by $V$.

\begin{Lemma}\label{cliff} We have the decomposition $\Fg=V\oplus[V,V]$ such that $[[V,V],V]\subset V$. As a vector space $[V,V]$ is isomorphic to 
$\Lambda^2V$ and concides 
with the span of $2vw-\beta(v,w)$ for all $v,w\in V$.
\end{Lemma}
\begin{proof} First, we compute the commutator 
$$[v,w]=vw-(-1)^{p(v)p(w)}wv=2vw-\beta(v,w).$$
Next we compute the commutator between $[v,w]$ and $u$ using super Leibniz identity
$$[u,[v,w]]=2[u,vw]=2([u,v]w+(-1)^{p(u)p(v)}v[u,w])=$$
$$2(2uvw-\beta(u,v)w+(-1)^{p(u)p(v)}2vuw-(-1)^{p(u)p(v)}\beta(u,w)v).$$
Using $vu=-(-1)^{p(u)p(v)}uv+\beta(v,u)$ and the symmetry of $\beta$ we obtain
$$[u,[v,w]]=2(\beta(u,v)w-(-1)^{p(u)p(v)}\beta(u,w)v).$$
Hence we have obtained $[[V,V],V]\subset V$ and by Jacobi identity $[[V,V],[V,V]]\subset [V,V]$.
\end{proof}

We concentrate on the case when $\beta$ is non-degenerate and $\dim V=(2p|2q)$, let $n=p+q$  and choose a basis $\xi_1,\dots,\xi_n,\eta_1,\dots,\eta_n$ such that
$$\beta(\xi_i,\xi_j)=\beta(\eta_i,\eta_j)=0,\quad \beta(\eta_i,\xi_j)=\delta_{i,j}.$$
The parity is defined by 
$$p(\xi_i)=p(\eta_i)=\begin{cases} 0\,\,\text{if}\,\,i\leq p\\ 1\,\,\text{if}\,\,i> p\end{cases}.$$
The corresponding Clifford superalgebra is isomorphic to $A^+_{p|q}$.
The defining relations are
$$\xi_i\xi_j+(-1)^{p(i)p(j)}\xi_j\xi_i=\eta_i\eta_j+(-1)^{p(i)p(j)}\eta_j\eta_i=0,$$
$$\eta_i\xi_j+(-1)^{p(i)p(j)}\xi_j\eta_i=\delta_{ij}.$$

\begin{Lemma} The Lie subsuperalgebra of $A^+_{p|q}$ generated by $\xi_i,\eta_i$ for $i=1,\dots,n$ is isomorphic to $\mathfrak{osp}(2p+1|2q)$.
\end{Lemma}
\begin{proof} In notations of Lemma \ref{cliff}, consider the adjoint action of $[V,V]$ on $V$. The Leibniz rule implies that
the form $\beta$ is invariant under this action. Hence $[V,V]$ is isomorphic to $\mathfrak{osp}(2p,2q)$ and $V$ is its natural representation.
Since obviously $V\oplus[V,V]$ is simple, it must be isomorphic to  $\mathfrak{osp}(2p+1|2q)$.
\end{proof}

\begin{Corollary} There exist homomorphisms of associative superalgebras $\pi_1:U(\mathfrak{osp}(2p|2q))\to A^+_{p|q}$ and
$\pi_2:U(\mathfrak{osp}(2p+1|2q))\to A^+_{p|q}$. 
\end{Corollary}

Let $q\neq 0$.
Let us assume that $e_1,\dots,e_n$ and $f_1,\dots f_n$ are the Chevalley generators of $\mathfrak{osp}(2p+1|2q)$ such that $e_p,f_p,e_n,f_n$ are odd and
all other generators are even. Then we have
$$\pi_2(e_i)=\begin{cases} \xi_i\eta_{i+1} & \text{if $i<n$}\\ \xi_n &\text{if $i=n$}\end{cases}\qquad
\pi_2(f_i)=\begin{cases} \xi_{i+1}\eta_{i} &\text{if $i<n$}\\ \eta_n &\text{if $i=n$}\end{cases},$$
and $\pi_1$ is obtained from $\pi_2$ by restriction.

Let $\Fg=\mathfrak{gl}(p|q)$, $\mathfrak{osp}(2p|2q)$ or $\mathfrak{osp}(2p+1|2q)$ and identify $\Z^m$ with the root lattice of $\Fg$ with basis 
consisting of the distinguished simple roots of $\Fg$. Let $\zeta:\Z^m \to \Z^n$ be the $\Z$-linear maps given 
by the matrices
\begin{equation}\label{eq:matrices}
\begin{bmatrix}
1  &    &        &  \\
-1 & 1  &        & \\
   & -1 &        &  \\
   &    & \ddots & 1 \\
   &    &        & -1
\end{bmatrix},
\qquad
\begin{bmatrix}
1  &    &        & & \\
-1 & 1  &        & &\\
   & -1 &        &  &\\
   &    & \ddots & 1 &\\
   &    &        & -1 & 2
\end{bmatrix},
\qquad
\begin{bmatrix}
1  &    &        & & \\
-1 & 1  &        & &\\
   & -1 &        &  &\\
   &    & \ddots & 1 &\\
   &    &        & -1 & 1
\end{bmatrix}
\end{equation}
respectively. Let $A^+_{p|q}=A_I$ be the Weyl algebra with index superset $I$, $I_{\bar 0}=\iv{1}{p}, I_{\bar 1}=\iv{p+1}{p+q}$.

\begin{Theorem}\label{thmClifford}
Let $p, q$ be nonnegative integers, $n=p+q>0$. Let $\Fg=\mathfrak{gl}(p|q)$, $\mathfrak{osp}(2p|2q)$ or $\mathfrak{osp}(2p+1|2q)$.
Then there is a 
commutative triangle of associative algebras with involution
\begin{equation} \label{eq:super-triangle}
\begin{tikzcd}[ampersand replacement=\&,
               column sep=small]
U\big(\mathfrak{g}\big) 
 \arrow[rr,"\pi"] 
 \arrow[rd,"\psi"'] \& \& A^+_{p|q} \\ 
\&  \mathcal{A}(\zeta)^+
 \ar[hookrightarrow,ur,"\varphi"'] \&
\end{tikzcd}
\end{equation}
where
$\varphi$ is given by Theorem \ref{thm:sufficient}, 
$\psi(e_i)=X_i$, $\psi(f_i)=Y_i$, $\psi(h_{ii})=\la_{ii}(u_i-1)$,
and
\begin{equation}
\pi(e_i)=
\begin{cases}
x_i\partial_{i+1}, & i<n,\\
x_n, & i=n,\; \mathfrak{g}=\mathfrak{osp}(2q+1|2p),\\
x^2_n, & i=n,\; \mathfrak{g}=\mathfrak{osp}(2q|2p),
\end{cases}\qquad
\pi(f_i)=\pi(e_i)^\ast,\;\pi(h_i)=x_i\partial_i-(-1)^{p(i)}\frac{1}{2},
\end{equation}
\end{Theorem}
The proof is similar to Theorem \ref{thmWeyl} and we leave it to the reader.

\subsection{On $A_{p|q}^+$ versus $A_{q|p}^-$}

If we disregard $\mathbb Z_2$-grading, then we have an isomorphisms of associative algebras $A^{\pm}_{p|0}\simeq A^{\mp}_{0|p}$.
We suspect that $A^+_{p|q}$ and $A^-_{q|p}$ are not isomorphic in general. Note also that
$A^-_{p|q}$ is isomorphic to the tensor product $M_{2^q}\otimes(A^-_{p|0})$, while $A^+_{q|p}$ is isomorphic to the supertensor product
$M_{2^q}\otimes(A^+_{0|p})$. However, we do have the following result.

\begin{Corollary}\label{corisom} Consider the sublattice $$\Gamma=\{(a_1,\dots,a_n)\,|\,a_1+\dots+a_n\in 2\Z\}$$ 
in $\Z^n$. Let $C^\pm_{p|q}$ denote the subsuperalgebra of elements of $A^\pm_{p|q}$ with the support in $\Gamma$.
Then $C^+_{p|q}$ and $C^-_{q|p}$ are isomorphic superalgebras. 
\end{Corollary}
\begin{proof} Theorem \ref{thmWeyl} and Theorem \ref{thmClifford} provide the homomorphisms from
$U(\mathfrak{osp}(2p|2q))$ to $A^-_{q|p}$ and $A^+_{p|q}$ respectively. It follows from formulas defining these isomorphisms that
$C^-_{q|p}$ and $C^+_{p|q}$ are respective images. Consider the modules
$$M^-:=A^-_{q|p}\otimes_{\K[\partial_1,\dots,\partial_n]}\K,\quad M^+:=A^+_{p|q}\otimes_{\K[\eta_1,\dots,\eta_n]}\K,$$
and let $$N^-=C^-_{q|p}(1\otimes 1),\quad N^+=C^+_{p|q}(1\otimes 1).$$
Note that $N^\pm$ is a simple module over $C^+_{p|q}$ and $C^-_{q|p}$, respectively, hence both $N^+$ and $N^-$ are simple $U(\mathfrak{osp}(2p|2q))$-modules.
Furthemore if $v=1\otimes 1$, then $$f_iv=0,\quad h_iv=-(-1)^{p(i)}v.$$ Thus both $N^+$ and $N^-$ are simple lowest weight modules with the same 
lowest weight. Thus, $N^+$ and $N^-$ are isomorphic, therefore they have the same annihilator $J\subset U(\mathfrak{osp}(2p|2q))$ and we obtain
$$C^+_{p|q}\simeq  U(\mathfrak{osp}(2p|2q)/J\simeq C^-_{q|p}.$$
\end{proof}

\subsection{Consequence for classical Lie algebras}

Taking $q=0$ in Theorem \ref{thmClifford} we immediately get the following result.

\begin{Corollary}\label{cor:oddcases} 
For $\Fg=\Fsl_n, \mathfrak{so}_{2n+1}$, or $\mathfrak{so}_{2n}$, there is a corresponding $\ga$ and a commutative triangle of associative algebras with involution
\begin{equation} \label{eq:odd-triangle}
\begin{tikzcd}[ampersand replacement=\&,
               column sep=small]
U\big(\mathfrak{g}\big) 
 \arrow[rr,"\pi"] 
 \arrow[rd,"\psi"'] \& \& A^+_{n|0} \\ 
\&  \mathcal{A}(\ga)^+
 \ar[hookrightarrow,ur,"\varphi"'] \&
\end{tikzcd}
\end{equation}
\end{Corollary}

We can now prove that further primitive quotients of enveloping algebras of classical Lie algebras are examples of TGWAs. 
This extends previous results by the authors \cite{HarSer2016}, were a condition for $U(\Fg)/J$ to be a not-necessarily abelian TGW algebra (i.e. we allowed 
$\si_i\si_j\neq \si_j\si_i$) was given.

\begin{Theorem} If $\mathfrak g=\mathfrak{so}_{2n},\mathfrak{so}_{2n+1}$ or $\mathfrak{sp}_{2n}$ and $M$
be a finite-dimensional completely pointed simple $\mathfrak g$-module and let $J=\Ann_{U(\Fg)}M$. Then $U(\Fg)/J$ is graded isomorphic to a TGWA of the form $\mathcal{A}(\ga)^+$.
The same is true for any fundamental representation of $\mathfrak{sl}_n$.
\end{Theorem}

\begin{proof}
The problem is to show that we can choose $\sigma_i$ so that the group $G$ generated by $\sigma_i$ is abelian.

If $\mathfrak g=\mathfrak{so}_{2n}$ or $\mathfrak{so}_{2n+1}$ and $M$ is a spinor representation, then
$U(\mathfrak g)/J$ is isomorphic to a subalgebra in the Clifford algebra with abelian
$G$ as follows from Corollary \ref{cor:oddcases}. 

Let $\mathfrak g=\mathfrak{sl}_n$. Consider the embedding
$\mathfrak{sl}_n\subset \mathfrak{so}_{2n+1}$ induced by the embedding of the corresponding Dynkin
diagrams. The restriction of the spinor representation to $\mathfrak{sl}_n$
contains all fundamental representations. Let $\ga$ be the rightmost matrix in \eqref{eq:matrices} and consider the subalgebra in
$\mathcal C\subset \mathcal{A}(\ga)^+$ generated by $X_1,\dots,X_{n-1},Y_1,\dots,Y_{n-1}$. Let
$I=\operatorname{Ann}_\mathcal C M$ and $\mathcal B=\mathcal C/I\simeq \operatorname{End}(M)$. Then $\mathcal B$ is a direct summand in the semisimple
algebra $\mathcal C$. Hence $\sigma_i$ for $i=1,\dots,n-1$ preserve $\mathcal B\cap R$ and the
statement follows. 

Let $\Gamma$ denote the set of weights of $M$.
Note that $\sigma_i$ must permute projectors $E_\beta$, hence it is
defined by a permuation of $\Gamma$.

Let $M$ be the standard representation of $\mathfrak{sp}_{2n}$. Then
$\Gamma=\{\pm\varepsilon_i\}$. Let
$\sigma_1=\sigma_2=\dots=\sigma_{n-1}$ be defined by the permutation
$\kappa=(\varepsilon_1,\dots,\varepsilon_n)(-\varepsilon_n,\dots,-\varepsilon_1)$
and $\sigma_n$ be defined by the permutation $\tau=(\varepsilon_1,-\varepsilon_1)\cdots(\varepsilon_n,-\varepsilon_n)$.

If $\mathfrak g=\mathfrak{so}_{2n}$ and $M$ is the standard representation, then we choose $\sigma_1=\dots=\sigma_{n-1}$
as in the previous case and  let $\sigma_n$ be given by the permutation
$\kappa\tau$.

Finally, if  $\mathfrak g=\mathfrak{so}_{2n+1}$ and $M$ is the standard
representation. Then $\Gamma= \{\pm\varepsilon_i,0\}$ and we define
$\sigma_1=\dots=\sigma_n$ by the permutation $(\varepsilon_1,\dots,\varepsilon_n,0,-\varepsilon_n,\dots,-\varepsilon_1)$.
\end{proof}

\section{Open problems}

\begin{Problem}
For a simple Lie algebra $\Fg$, list all finite-dimensional irreducible $\Fg$-modules $M$ for which there exists a graded isomorphism between $U(\Fg)/\Ann_{U(\Fg)}M$ and a TGW algebra (equivalently, for which there is a choice of commuting $\sigma_i$). We believe none of the non-fundamental representations of $\mathfrak{sl}_n$ for $n>2$ are in this list. The remaining cases to consider are the $27$-dimensional representation of $E_6$ and $56$-dimensional representation of $E_7$.
\end{Problem}

\begin{Problem}\label{prb:consistency}
Find necessary and sufficient conditions for a not necessarily regular TGW algebra $\TGWA{\mu}{R}{\si}{t}$ to be consistent, generalizing the main result of \cite{FutHar2012b}.
\end{Problem}

\bibliographystyle{siam}

\end{document}